\documentclass[12pt,reqno]{amsart}
\usepackage{amssymb}
\usepackage{epsfig}
\usepackage{mathabx}
\usepackage{txfonts}
\usepackage{amsmath}
\usepackage{amssymb}
\usepackage{color}
\usepackage{amsmath}
\usepackage{cite}

\numberwithin{equation}{section}

\begin{document}

\newtheorem{mthm}{Theorem}
\newtheorem{mcor}{Corollary}
\newtheorem{mpro}{Proposition}
\newtheorem{mfig}{figure}
\newtheorem{mlem}{Lemma}
\newtheorem{mdef}{Definition}
\newtheorem{mrem}{Remark}
\newtheorem{mpic}{Picture}
\newtheorem{rem}{Remark}[section]
\newcommand{\ra}{{\mbox{$\rightarrow$}}}
\newtheorem{Remark}{Remark}[section]
\newtheorem{thm}{Theorem}[section]
\newtheorem{pro}{Proposition}[section]
\newtheorem*{proA}{Proposition A}

\newtheorem{lem}{Lemma}[section]
\newtheorem{defi}{Definition}[section]
\newtheorem{cor}{Corollary}[section]

\title[]{Extremal solution and Liouville theorem for anisotropic elliptic equations}

\author{Yuan Li }
\address{School of Mathematics, Hunan University, Changsha 410082,
                 PRC  }
\email{liy93@hnu.edu.cn}

\thanks{}

\date{}

\maketitle

\begin{abstract}
We study the quasilinear Dirichlet boundary problem
 \begin{equation}\nonumber
\left\{
\begin{aligned}
-Qu&=\lambda e^{u} \indent \mbox{in}\indent\Omega\\
u&=0 \indent \mbox{on}\indent\partial\Omega,\\
\end{aligned}
\right.
\end{equation}
where $\lambda>0$ is a parameter, $\Omega\subset\mathbb{R}^{N}$ with $N\geq2$ be a bounded domain, and the operator $Q$, known as Finsler-Laplacian or anisotropic Laplacian, is defined by
$$Qu:=\sum_{i=1}^{N}\frac{\partial}{\partial x_{i}}(F(\nabla u)F_{\xi_{i}}(\nabla u)). $$
Here,  $F_{\xi_{i}}=\frac{\partial F}{\partial\xi_{i}}$ and $F: \mathbb{R}^{N}\rightarrow[0,+\infty)$ is a convex function of  $ C^{2}(\mathbb{R}^{N}\setminus\{0\})$, that satisfies certain assumptions. We derive the existence of extremal solution and obtain that it's regular, if $N\leq9$.

We also concern the H\'{e}non type anisotropic Liouville equation, namely,
$$-Qu=(F^{0}(x))^{\alpha}e^{u}\indent\mbox{in}\indent\mathbb{R}^{N},$$
where $\alpha>-2$, $N\geq2$ and $F^{0}$ is the support function of $K:=\{x\in\mathbb{R}^{N}:F(x)<1\}$ which is defined by
$$F^{0}(x):=\sup_{\xi\in K}\langle x,\xi\rangle.$$
We obtain the Liouville theorem for stable solutions and the finite Morse index solutions for $2\leq N<10+4\alpha$ and $3\leq N<10+4\alpha^{-}$ respectively, where $\alpha^{-}=\min\{\alpha,0\}$.

\end{abstract}


\section{Introduction and main results}
We consider the quasilinear elliptic equation
\begin{align}\label{1.1}
-Qu=f(x,u)\indent \mbox{in}\indent\Omega\subset\mathbb{R}^{N},
\end{align}
where $N\geq2$, and the operator $Q$ is called anisotropic Laplacian or Finsler Laplacian, which is defined by
$$Qu:=\sum_{i=1}^{n}\frac{\partial}{\partial x_{i}}(F(\nabla u)F_{\xi_{i}}(\nabla u)),$$
where $F_{\xi_{i}}=\frac{\partial F}{\partial\xi_{i}}$ and $F:\mathbb{R}^{n}\rightarrow[0, \infty)$ is a convex and homogeneous function of $C^{2}(\mathbb{R}^{n}\setminus\{0\})$. $F^{0}$ is the support function of $K:=\{x\in\mathbb{R}^{N}:F(x)<1\}$ which is defined by
$$F^{0}(x):=\sup_{\xi\in K}\langle x,\xi\rangle.$$

 Especially, when $F(\xi)=|\xi|$, that is the isotropic case, the operator $Q$ becomes the classical Laplacian operator. There is a long history to research the anisotropic operator. In the early twentieth century, Wulff \cite{W} used such operators to study crystal shapes and minimization of anisotropic surface tensions. For more applications of anisotropic operator in the case of crystallization refer to \cite{AT,ATW}. Furthermore, for more literature on anisotropic operator, we refer interested readers to \cite{AFTL,CS09,CFV,cfv,FV14,FL,FM,WX,WX12} and  references therein.

For subsequent calculation, we  give some certain assumptions to the function $F$. Let $F: \mathbb{R}^{N}\rightarrow[0,+\infty)$ be a convex function in  $ C^{2}(\mathbb{R}^{N}\setminus\{0\})$ where $F(t\xi)=|t|F(\xi)$ for any $t\in\mathbb{R}$ and $\xi\in\mathbb{R}^{N}$. We assume that $F(\xi)>0$ for any $\xi\neq0$ and there exist constant $0<a\leq b<\infty$ and $0<\lambda\leq\Lambda<\infty$ such that
\begin{equation*}
a|\xi|\leq F(\xi)\leq b|\xi|  \indent\mbox{for any}\indent \xi\in\mathbb{R}^{n},
\end{equation*}
and
\begin{align}\label{1.2}
\lambda^{2}|V|^{2}\leq F_{\xi_{i}\xi_{j}}(\xi)V_{i}V_{j}\leq\Lambda|V|^{2},
\end{align}
for any $\xi\in\mathbb{R}^{n}$ and $V\in\xi^{\bot}$ where $\xi^{\bot}:=\{V\in\mathbb{R}^{n}: \langle V, \xi\rangle=0\}$.  Let $B_{r}(x_{0}):=\{x\in\mathbb{R}^{n}: F^{0}(x-x_{0})<r\}$ denote a Wulff ball of radius $r$ with center at $x_{0}$, and we  use this notation $B_{r}(x_{0})$ throughout the paper. Set $\kappa_{0}:=|B_{1}(x_{0})|$, where $|B_{1}(x_{0})|$ is the Lebesgue measure of $B_{1}(x_{0})$. It follows from the assumptions of $F$, the following properties holds, we refer interested readers to \cite{AFTL,FL, FK, WX, WX12}.
\begin{pro}\label{pro1.1}
We have the following properties:

(1) $|F(x)-F(y)|\leq F(x+y)\leq F(x)+F(y)$;

(2) $|\nabla F(x)|\leq C$ for any $x\neq0$;

(3) $\langle \xi, \nabla F(\xi)\rangle=F(\xi)$, $\langle x, \nabla F^{0}(x)\rangle=F^{0}(x)$ for any $x\neq0$, $\xi\neq0$;

(4) $\sum_{j=1}^{N}F_{\xi_{i}\xi_{j}}(\xi)\xi_{j}=0$, for any $i=1,2,\cdots,N$;

(5) $F(\nabla F^{0}(x))=1$, $F^{0}(\nabla F(x))=1$;

(6) $F_{\xi_{i}}(t\xi)=sgn(t)F_{\xi_{i}}(\xi)$;

(7) $F^{0}(x)F_{\xi}(\nabla F^{0}(x))=x$.
\end{pro}

In order to state  our main results, we give the definition of the weak stable solutions of (\ref{1.1}).
\begin{defi}
We say that $u$ is a weak solution of (\ref{1.1}), if $u\in H^{1}_{loc}(\Omega)$, $f(x,u)\in L^{1}_{loc}(\Omega)$, and the following holds
$$\int_{\Omega}F(\nabla u)F_{\xi}(\nabla u)\cdot\nabla\phi dx=\int_{\Omega}f(x,u)\phi dx, $$
for all $\phi\in C_{c}^{\infty}(\Omega)$. To go further, we say that the weak solution is stable, if for all $\phi\in C_{c}^{\infty}(\Omega)$ holds,
\begin{align}
\int_{\Omega}F_{\xi_{i}}(\nabla u)F_{\xi_{j}}(\nabla u)\phi_{x_{i}}\phi_{x_{j}}+F(\nabla u)F_{\xi_{i}\xi_{j}}(\nabla u)\phi_{x_{i}}\phi_{x_{j}}-\partial_{u}f(x,u)\phi^{2}dx\geq0.\nonumber
\end{align}
\end{defi}

If $f(x,u)=\lambda e^{u}$ for some positive parameter $\lambda$, we consider the following Dirichlet  boundary problem
 \begin{equation}
\left\{
\begin{aligned}\label{1.3}
-Qu&=\lambda e^{u} \indent \mbox{in}\indent\Omega\\
u&=0 \indent \mbox{on}\indent\partial\Omega,\\
\end{aligned}
\right.
\end{equation}
where $\Omega$ is a bounded domain. For this equation, we say that the solution $\underline{u}$ is minimal, if for any supersolution $u$ satisfies the following equation
 \begin{equation}
\left\{
\begin{aligned}\nonumber
-Qu&\geq\lambda e^{u} \indent \mbox{in}\indent\Omega\\
u&\geq0 \indent \mbox{on}\indent\partial\Omega,\\
\end{aligned}
\right.
\end{equation}
it holds $u\geq\underline{u}$. We also say that the solution $u\in W_{0}^{1,2}(\Omega)$ is regular if and only if $e^{u}\in L^{\infty}(\Omega)$, and is singular iff $e^{u}\in L^{1}(\Omega)$.

For the isotropic case, i.e. $F(\xi)=|\xi|$, the equation becomes
 \begin{equation}
\left\{
\begin{aligned}\label{2 lambda}
-\Delta u&=\lambda e^{u} \indent \mbox{in}\indent\Omega\\
u&=0 \indent \mbox{on}\indent\partial\Omega,\\
\end{aligned}
\right.
\end{equation}
it's well known that there exist a maximal parameter $\lambda^{*}>0$ such that for any $0<\lambda<\lambda^{*}$, the equation (\ref{2 lambda}) admits a minimal solution $u_{\lambda}$, and if $\lambda\rightarrow\lambda^{*}$, $\lambda<\lambda^{*}$ the solution $u_{\lambda}$ converges to the so-called extremal solution, which turns out to be smooth for $N\leq9$, we refer interested reader to \cite{CR,MP}. Moreover, Garcia Azorero et al. in \cite{GP,GPP} extend these results to the p-Laplacian i.e. the equation
 \begin{equation}\nonumber
\left\{
\begin{aligned}
-\Delta_{p} u&\equiv-div(|\nabla u|^{p-2}\nabla u)=\lambda e^{u} \indent \mbox{in}\indent\Omega\\
u&=0 \indent \mbox{on}\indent\partial\Omega.\\
\end{aligned}
\right.
\end{equation}
Inspired by these results, we consider the existence of extremal solution to the equation (\ref{1.3}), the following are our main results.

\begin{thm}\label{thm1.1}
There exist a constant $\lambda_{0}$ such that if $0<\lambda<\lambda_{0}$, then (\ref{1.3}) has a minimal solution which is regular and stable.
\end{thm}
For the isotropic case, it's well known that equation (\ref{2 lambda}) has no solution if $\lambda$ is bigger than the first eigenvalue for the Laplacian. Similarly, we have the following nonexistence result.
\begin{thm}\label{thm1.2}
There is no solution for equation (\ref{1.3}) if $\lambda>\lambda_{1}$, where $\lambda_{1}$ is the first eigenvalue for the Finsler Laplacian.
\end{thm}

From the above two theorems, we can define
$$\lambda^{*}:=\sup\{\lambda| (\ref{1.3})\indent\mbox{ has solution}\}.$$
This leads us naturally to ask what will happen at $\lambda^{*}$?

\begin{thm}\label{thm1.3}
Let $\{\lambda_{n}\}$ be an increasing sequence such that $\lambda_{n}\rightarrow\lambda^{*}$, and $\underline{u}_{n}=\underline{u}_{n}(\lambda_{n})$ be the corresponding minimal solution, then, we have
$$\underline{u}_{n}\rightarrow u^{*}\indent\mbox{strongly in } W_{0}^{1,2}(\Omega),$$
$$e^{\underline{u}_{n}}\rightarrow e^{u^{*}}\indent\mbox{strongly in } L^{\frac{2^{*}}{2^{*}-1}}(\Omega)$$
where $2^{*}=\frac{2N}{N-2}$, $u^{*}$ is a singular solution of (\ref{1.3}) with the parameter $\lambda^{*}$, we call it as extremal solution.
\end{thm}

\begin{rem}
Since the minimal solution $\underline{u}_{n}$ is stable, so it's easy to see that the extremal solution $u^{*}$ is also stable.
\end{rem}

It can be seen from the Theorem \ref{thm1.1} that the minimal solution is regular, so, what about the smoothness of the extremal solution? In isotropic case, the extremal solution is regular provided that $N\leq9$ and for the p-Laplacian case when $N<\frac{4p}{p-1}+p$. For our case, we have the following conclusion.

\begin{thm}\label{thm1.4}
If $N\leq9$, then the extremal solution of (\ref{1.3}) is regular.
\end{thm}

\begin{rem}
For the partial regularity result of extremal solution of (\ref{1.3}), the Hausdorff dimension of the singular set does not exceed $N-10$,  we refer to \cite{FL}.
\end{rem}

We also have the strong interest to establish liouville type theorems. We now consider $\Omega$ to be the entire space $\mathbb{R}^{N}$, if $f(x,u)=(F^{0}(x))^{\alpha}e^{u}$ where $\alpha>-2$, the equation becomes to

\begin{align}\label{1.5}
-Qu=(F^{0}(x))^{\alpha}e^{u}\indent\mbox{in}\indent\mathbb{R}^{N},
\end{align}
we concern the Liouville theorems for stable solutions and the finite Morse index solutions of this equation. Now, let us first recall some classic Liouville theorems of equation (\ref{1.5}) for the isotropic case, and the corresponding equation becomes
$$-\Delta u=|x|^{\alpha}e^{u}\indent\mbox{in}\indent\mathbb{R}^{N},$$
which is called H\'{e}non type Liouville equation. If $\alpha=0$, Farina in \cite{F} proved that there is no classical stable solution when $2\leq N\leq9$, and Dancer with Farina \cite{DF} obtain that if $3\leq N\leq9$, there is no classical solution which is stable outside a compact set of $\mathbb{R}^{N}$. Later Wang and Ye \cite{WY} extend their results to the case of $\alpha>-2$, and obtain the counterpart results for $2\leq N<10+4\alpha$ and $3\leq N<10+4\alpha^{-}$, where $\alpha^{-}=\min\{\alpha,0\}$, respectively. Recently, Ao and Yang \cite{AY} obtain some related results to the cosmic strings equation. For anisotropic case, we assume that the function $F$ in operator $Q$ satisfies

 \begin{align}\label{1.6}
 \langle F_{\xi}(x),F_{\xi}^{0}(y)\rangle=\frac{\langle x,y\rangle}{F(x)F^{0}(y)}\indent\mbox{for all }\indent x,y\in\mathbb{R}^{N},
 \end{align}
 where such an assumption was first introduced by Ferone and Kawohl in \cite{FK}. Recently, Fazly and Li \cite{FL} under this assumption derive the Liouville theorem for stable solutions and the finite Morse index solutions. So, a nature problem arises: can we extend the results of Wang and Ye \cite{WY} to anisotropic case? In this paper we give a positive answer.

The following are our main results, the first one is the Liouville theorem for stable solutions.
\begin{thm}\label{thm1.5}
For $\alpha>-2$, if $2\leq N<10+4\alpha$, then there is no weak stable solution to equation (\ref{1.5}).
\end{thm}

Next, we state the result of Liouville theorem for the finite Morse index solutions.
\begin{thm}\label{thm1.6}
Let $\alpha>-2$, under the assumption of (\ref{1.6}), if $3\leq N<10+4\alpha^{-}$, then equation (\ref{1.5}) admits no weak solutions which is stable outside a compact set.
\end{thm}
The article is organized as follows. In section \ref{sec2}, we review some classic results, which play a key role in the proof of our main theorem. In section \ref{sec3}, we prove the main theorems Theorem \ref{thm1.1}-\ref{thm1.4}. In the last section, section \ref{sec4}, we give the proof of Liouville theorems Theorem \ref{thm1.5}, \ref{thm1.6}.

\section{Preliminaries}\label{sec2}
In this section, we first review some of the classic results, which are crucial in our paper. The first one is the weak and strong maximum principle, we refer to \cite{FK,S1} and references therein.

\begin{pro}\label{+pro2.1}
Suppose $u,v\in W^{1,2}(\Omega)$ satisfies $-Qu\leq-Qv$ in $\Omega$ and $u\leq v$ on $\partial\Omega$. Then $u\leq v$ in $\Omega$. Moreover, $u\equiv v$ or $u<v$ in $\Omega$.
\end{pro}

Next, we will prove the $L^{\infty}$-regularity, the proof is similar to \cite{FS}, with some small modifications. For the convenience of readers, we carry it out here. Before proving this result, we first state the following conclusion, given in \cite{S2}.

\begin{pro}\label{+pro2.2}
Assume that $\phi: [0,\infty)\rightarrow[0,\infty)$ is a nonincreasing function such that if $h>k>k_{0}$, for some $\alpha>0$, $\beta>1$ and $\phi(h)\leq \frac{C(\phi(k))^{\beta}}{(h-k)^{\alpha}}$, then $\phi(k_{0}+d)=0$, where $d^{\alpha}=C2^{\frac{\alpha\beta}{\beta-1}}\phi(k_{0})^{\beta-1}$.
\end{pro}

\begin{lem}\label{+lem2.1}
Let $u\in W_{0}^{1,2}(\Omega)$ be a solution of
 \begin{equation}\nonumber
\left\{
\begin{aligned}
-Qu&=f \indent \mbox{in}\indent\Omega\\
u&=0 \indent \mbox{on}\indent\partial\Omega.\\
\end{aligned}
\right.
\end{equation}
where $\Omega\subset\mathbb{R}^{N}$ be a bounded domain, $f\in L^{p}(\Omega)$ with $p>\frac{N}{2}$, then $u\in L^{\infty}(\Omega)$.
\end{lem}

\begin{proof}
Let $u_{k}=sign(u)(|u|-k)_{+}$ where $(|u|-k)_{+}=max\{|u|-k,0\}$ for some nonnegative constant $k$. Define $\Omega_{k}=\{x\in\Omega: |u|>k\}$, then, we have
$$\int_{\Omega_{k}}F^{2}(\nabla u_{k})dx=\int_{\Omega_{k}}fu_{k}dx\leq\left(\int_{\Omega_{k}}|u_{k}|^{2^{\ast}}dx\right)^{\frac{1}{2^{\ast}}}\left(\int_{\Omega_{k}}|f|^{p}dx\right)^{\frac{1}{p}}|\Omega_{k}|^{1-(\frac{1}{p}+\frac{1}{2^{\ast}})},$$
where $2^{\ast}=\frac{2N}{N-2}$ by sobolev inequality, we obtain
$$\left(\int_{\Omega_{k}}|u_{k}|^{2^{\ast}}dx\right)^{\frac{1}{2^{\ast}}}\leq C\left(\int_{\Omega_{k}}F^{2}(\nabla u_{k})dx\right)^{\frac{1}{2}},$$
it follows that
$$\left(\int_{\Omega_{k}}|u_{k}|^{2^{\ast}}dx\right)^{\frac{1}{2^{\ast}}}\leq C\left(\int_{\Omega_{k}}|f|^{p}dx\right)^{\frac{1}{p}}|\Omega_{k}|^{1-(\frac{1}{p}+\frac{1}{2^{\ast}})}.$$
For $0<k<h$, we have $\Omega_{h}\subset\Omega_{k}$ and
$$|\Omega_{h}|^{\frac{1}{2^{\ast}}}(h-k)=\left(\int_{\Omega_{h}}(h-k)^{2^{\ast}}\right)^{\frac{1}{2^{\ast}}}\leq\left(\int_{\Omega_{k}}|u_{k}|^{2^{\ast}}\right)^{\frac{1}{2^{\ast}}}.$$
Hence, we have
$$|\Omega_{h}|\leq\frac{C}{(h-k)^{2^{\ast}}}\left(\int_{\Omega_{k}}|f|^{p}dx\right)^{\frac{2^{\ast}}{p}}|\Omega_{k}|^{2^{\ast}[1-(\frac{1}{p}+\frac{1}{2^{\ast}})]},$$
since, $p>\frac{N}{2}$, we have $2^{\ast}[1-(\frac{1}{p}+\frac{1}{2^{\ast}})]>1$, let $\phi(h)=|\Omega_{h}|$, $\alpha=2^{\ast}$, $\beta=2^{\ast}[1-(\frac{1}{p}+\frac{1}{2^{\ast}})]$ and $k_{0}=0$, we have
$$\phi(h)\leq\frac{C}{(h-k)^{\alpha}}\phi(k)^{\beta}\indent\mbox{for any }h>k>0,$$
it follows from Proposition, we have $\phi(d)=0$, where $d=C2^{\frac{\alpha\beta}{\beta-1}}\phi(k_{0})^{\beta-1}$, then $\parallel u\parallel_{L^{\infty}(\Omega)}<\infty$.

\end{proof}

Wang and Xia \cite{WX12} proved the following famous Moser-Trudinger inequality in an anisotropic version.

\begin{pro}\label{+pro2.3}
Let $\Omega$ be a bounded domain in $\mathbb{R}^{N}$, $N\geq2$. Let $u\in W_{0}^{1,N}(\Omega)$, then there exist a constant $C(N)$ such that
$$\int_{\Omega}\exp\left[\frac{\beta|u|}{\parallel F(\nabla u)\parallel_{L^{N}}}\right]dx\leq C(N)|\Omega|,$$
where $\beta\leq\beta_{N}=N^{\frac{N}{N-1}}\kappa_{0}^{\frac{1}{N-1}}$, $\beta_{N}$ is optimal in the sense that if $\beta>\beta_{N}$, we can find a sequence $\{u_{k}\}$ such that  $\int_{\Omega}\exp\left[\frac{\beta|u_{k}|}{\parallel F(\nabla u_{k})\parallel_{L^{N}}}\right]dx$ diverges. Moreover, by Young's inequality, there exists positive constants $c_{1}$, $c_{2}$ depending only on $N$, such that for any $\gamma>0$,
\begin{align}\label{+2.1}
\int_{\Omega}e^{\gamma |u|}dx\leq c_{1}|\Omega|\exp(c_{2}\gamma^{N}\parallel F(\nabla u)\parallel_{L^{N}}^{N}).
\end{align}
\end{pro}

The following result is Mountain-Pass lemma, which was first put forward by Ambrosetti and Rabinowitz \cite{AR}, and it has attracted a lot of many mathematicians to study the critical point theory, see \cite{BN,BN1,NT} and reference therein.

\begin{pro}\label{+pro2.4}
Let $X$ be a real Banach space and $I\in C^{1}(X,\mathbb{R})$. Suppose there exist two real number $a$ and $R$, $R>0$, such that
$$I(x)\geq a\indent\mbox{on the sphere }\indent S_{R}=\{x\in X: \parallel x\parallel=R\},$$
$$I(0)<a\indent\mbox{and}\indent I(e)<a\indent\mbox{for some }e\mbox{ with }\parallel e\parallel>R.$$
Let $\Gamma$ denote the class of continuous paths joining $0$ and $e$, that is,
$$\Gamma=\{g\in C([0,1], X)| g(0)=0, g(1)=e\},$$
set
$$c=\inf_{g\in\Gamma}\max_{t\in[0,1]}I(g(t)),$$
then $c\geq a$, and exist Palais-Smale sequence, that is, exist sequence $\{x_{n}\}\subset X$ such that
$$I(x_{n})\rightarrow c\indent\mbox{and}\indent I'(x_{n})\rightarrow0.$$
Moreover, if $I$ satisfies the Palais-Smale compactness condition, namely, for any  Palais-Smale sequence contains a convergent subsequence. Then $c$ is a critical value of $I$.
\end{pro}

Finally, we introduce the first eigenvalue result for the Finsler Laplacian, see \cite{BFK,DG}.

\begin{pro}\label{+pro2.5}
There exists the first eigenvalue of
 \begin{equation}\nonumber
\left\{
\begin{aligned}\label{1 lambda}
-Qu&=\lambda u \indent \mbox{in}\indent\Omega\\
u&=0 \indent \mbox{on}\indent\partial\Omega,\\
\end{aligned}
\right.
\end{equation}
namely $\lambda(\Omega)>0$, and it is simple. Moreover, the first eigenfunctions have a sign and belong to $C^{1,\alpha}$. Finally, the following variational formulation holds:
$$\lambda(\Omega)=\min_{u\in W_{0}^{1,2}(\Omega)\setminus\{0\}}\frac{\int_{\Omega}F^{2}(\nabla u)dx}{\int_{\Omega}|u|^{2}dx}.$$
\end{pro}

\section{Existence of extremal solution}\label{sec3}
In this section, we will give a detailed proof of our main Theorem \ref{thm1.1}-\ref{thm1.4}. At first, we use the iterative method to prove Theorem \ref{thm1.1}.

\begin{proof}[Proof of Theorem \ref{thm1.1}]
Consider the equation
 \begin{equation}
\left\{
\begin{aligned}\nonumber
-Qw&=1 \indent \mbox{in}\indent B\\
w&=0 \indent \mbox{on}\indent\partial B.\\
\end{aligned}
\right.
\end{equation}
where $B$ is a ball such that $\overline{\Omega}\subset B$. Easy to see that $w\in C^{1,\alpha}(\Omega)$ and $w>0$ on $\partial\Omega$. Set $M=\max_{x\in\overline{\Omega}}w(x)$, let $\lambda_{0}=e^{-M}$ take $0<\lambda<\lambda_{0}$, we have
 \begin{equation}
\left\{
\begin{aligned}\nonumber
-Qw&=1>\lambda e^{w} \indent \mbox{in}\indent \Omega\\
w&>0 \indent \mbox{on}\indent\partial \Omega.\\
\end{aligned}
\right.
\end{equation}
so, $w$ is a supersolution of (\ref{1.3}). Let $u_{1}$ be a solution of
 \begin{equation}
\left\{
\begin{aligned}\nonumber
-Qu_{1}&=\lambda \indent \mbox{in}\indent\Omega\\
u&=0 \indent \mbox{on}\indent\partial\Omega,\\
\end{aligned}
\right.
\end{equation}
by weak comparison principle, we get $0\leq u_{1}\leq w$. Next, take the solutions of
 \begin{equation}
\left\{
\begin{aligned}\nonumber
-Qu_{n}&=\lambda e^{u_{n-1}} \indent \mbox{in}\indent\Omega\\
u_{n}&=0 \indent \mbox{on}\indent\partial\Omega.\\
\end{aligned}
\right.
\end{equation}
We obtain an increasing sequence $\{u_{n}\}$ satisfies $0\leq u_{n}\leq w$,  let $\underline{u}=\lim_{n}u_{n}$, it's clear that $\underline{u}$ is minimal and regular.

Next, we will prove that minimal solution is stable. Define
$$K=\{v\in W_{0}^{1,2}(\Omega)|0\leq v\leq \underline{u} \},$$
since the energy functional
$$E(u)=\frac{1}{2}\int_{\Omega}F^{2}(\nabla u)dx-\lambda\int_{\Omega}e^{u}dx,$$
so, it's clear that there exist $u\in K$ such that
$$\min_{v\in K}E(v)=E(u).$$
We want to claim the minimizer $u\in K$ is $\underline{u}$. By definition of $K$, it suffices to prove that $\underline{u}\leq u$. The minimal solution is obtained as the limit of an increasing sequence $\{u_{n}\}$ of
  \begin{equation}\label{3.1}
\left\{
\begin{aligned}
-Qu_{n}&=\lambda e^{u_{n-1}} \indent \mbox{in}\indent \Omega\\
u_{n}&=0 \indent \mbox{on}\indent\partial \Omega,\\
\end{aligned}
\right.
\end{equation}
with $u_{0}=0$, on other hand $u\geq0=u_{0}$, so, we have $e^{u_{0}}\leq e^{u}$. Moreover, by the definition of the minimizer $u$, we have
$$\langle -Qu, (u-u_{1})^{-}\rangle\geq\lambda\langle e^{u}, (u-u_{1})^{-}\rangle,$$
where $(u-u_{1})^{-}=min\{0, u-u_{1}\}$. Meanwhile, it follows from the equation (\ref{3.1}), we have
$$\langle -Qu_{1}, (u-u_{1})^{-}\rangle=\lambda\langle e^{u_{0}}, (u-u_{1})^{-}\rangle,$$
then
\begin{align}
C\int_{\{u\leq u_{1}\}}|\nabla(u-u_{1})|^{2}dx&\leq\int_{\{u\leq u_{1}\}}\left(F(\nabla u_{1})F_{\xi}(\nabla u_{1})-F(\nabla u)F_{\xi}(\nabla u)\right)\nabla(u_{1}-u)dx\nonumber\\
&=\langle -Qu_{1}+Qu, (u-u_{1})^{-}\rangle\nonumber\\
&\leq\lambda\int_{\{u\leq u_{1}\}}(e^{u_{0}}-e^{u_{1}})(u_{1}-u)dx\leq0,\nonumber
\end{align}
this implies that $u_{1}\leq u$. Similarly, we can obtain that $u_{n}\leq u$, thus we have $u\geq \underline{u}$.

The strong maximum principle tell us that $\underline{u}>0$ in $\Omega$, so for any $0\leq\phi\in C_{c}^{\infty}(\Omega)$, there exist $\varepsilon>0$ such that $0\leq\underline{u}-\varepsilon\phi\in K$, since $\underline{u}$ is minimizer in $K$, we have $E(\underline{u})\leq E(\underline{u}-\varepsilon\phi)$. It follows from that Talyor extension, we have
$$0\leq E(\underline{u}-\varepsilon\phi)-E(\underline{u})=\frac{\varepsilon^{2}}{2}\langle E^{''}(\underline{u})\phi,\phi\rangle+o(\varepsilon^{2}),$$
therefore
$$\langle E^{''}(\underline{u})\phi,\phi\rangle\geq0,$$
by density, we have
$$\langle E^{''}(\underline{u})v,v\rangle\geq0,$$
for all $v\in W_{0}^{1,2}(\Omega)$.
\end{proof}

From the proof of the above theorem, it can be seen that if the supersolution of the equation (\ref{1.3}) is regular, then the minimal solution derived from it is also regular. Actually, if the supersolution is singular, we can also obtain that the minimal solution is regular.

\begin{thm}
Assume that $u_{0}\in W_{0}^{1,2}(\Omega)$, $e^{u_{0}}\in L^{1}(\Omega)$, is a singular solution of
 \begin{equation}
\left\{
\begin{aligned}\nonumber
-Qu_{0}&=\widetilde{\lambda} e^{u_{0}} \indent \mbox{in}\indent \Omega\\
u_{0}&=0 \indent \mbox{on}\indent\partial \Omega.\\
\end{aligned}
\right.
\end{equation}
then, for any $\lambda\in(0, \widetilde{\lambda})$ the problem
 \begin{equation}\label{3.2}
\left\{
\begin{aligned}
-Qu&=\lambda e^{u} \indent \mbox{in}\indent \Omega\\
u&=0 \indent \mbox{on}\indent\partial \Omega.\\
\end{aligned}
\right.
\end{equation}
has a minimal regular solution.
\end{thm}

\begin{proof}
The same as the proof of Theorem \ref{thm1.1}, we can also prove the existence of minimal solution by iterative method. I will not repeat it here. In the following, we will prove the minimal solution is regular.

If $N=2$, it follows from Lemma \ref{+lem2.1} and Moser Trudinger inequality (\ref{+2.1}), the conclusion is obvious.

 For $N>2$, since $u_{0}$ is a singular supersolution, consider
 \begin{equation}
\left\{
\begin{aligned}\nonumber
-Qu_{1}&=\lambda e^{u_{0}} \indent \mbox{in}\indent \Omega\\
u_{1}&=0 \indent \mbox{on}\indent\partial \Omega.\\
\end{aligned}
\right.
\end{equation}
we have $u_{1}=\frac{\lambda}{\widetilde{\lambda}}u_{0}$, which satisfies $u_{1}<u_{0}$, we have $e^{u_{1}}\in L^{\frac{\widetilde{\lambda}}{\lambda}}(\Omega)$. In particular $u_{1}\in L^{r}(\Omega)$ for any $r\in(1,\infty)$. Let $u_{2}$ be a solution of
 \begin{equation}
\left\{
\begin{aligned}\nonumber
-Qu_{2}&=\lambda e^{u_{1}} \indent \mbox{in}\indent \Omega\\
u_{2}&=0 \indent \mbox{on}\indent\partial \Omega.\\
\end{aligned}
\right.
\end{equation}
then we find a weak solution $u_{2}\in W_{0}^{1,2}(\Omega)$, by the comparison principle, we have $0<u_{2}\leq u_{1}<u_{0}$. Therefore, $u_{2}\in L^{r}(\Omega)$ for any $r\in(1,\infty)$, to go further $u_{2}e^{u_{1}}\in L^{1}(\Omega)$. Consider the function $f(s)=e^{sx_{0}}$, for $0<t<1$, we have
$$e^{tx_{0}}+(1-t)x_{0}e^{tx_{0}}\leq e^{x_{0}},$$
taking $t=\frac{\lambda}{\widetilde{\lambda}}$ and $x_{0}=u_{0}$, we have
$$e^{u_{1}}+(1-t)u_{0}e^{u_{1}}\leq e^{u_{0}},$$
since $u_{2}\leq u_{1}$, we have
$$e^{u_{1}}+(1-t)u_{0}e^{u_{2}}\leq e^{u_{0}}.$$
since
$$-Q(\frac{1}{2}u_{2}^{2})=-u_{2}Qu_{2}-F^{2}(\nabla u)\leq\lambda u_{2}e^{u_{1}},$$
hence
$$-Q(\frac{1}{2}(1-t)u_{2}^{2})\leq\lambda(1-t) u_{2}e^{u_{1}}\leq \lambda e^{u_{0}}=-Qu_{1}.$$
On the other hand, since $-Q(\frac{1}{2}u_{2}^{2})=-u_{2}Qu_{2}-F^{2}(\nabla u)\in L^{1}(\Omega)$, define
$$ w_{k}=\left\{
\begin{aligned}
\frac{1}{2}u_{2}^{2}, &  & \indent\mbox{if}\indent \frac{1}{2}u_{2}^{2}<k \\
k, & & \indent\mbox{if}\indent \frac{1}{2}u_{2}^{2}\geq k, \\
\end{aligned}
\right.
$$
since, $e^{u_{1}}\in L^{\frac{\widetilde{\lambda}}{\lambda}}$ and $u_{2}\in L^{r}$ for any $r\in(1,\infty)$, we have
$$\int_{\Omega}F^{2}(\nabla w_{k})dx\leq \lambda\int_{\Omega}w_{k}u_{2}e^{u_{1}}dx\leq\frac{\lambda}{2}\int_{\Omega}u_{2}^{3}e^{u_{1}}dx<\infty,$$
let $k\rightarrow\infty$, we have $\frac{1}{2}u_{2}^{2}\in W_{0}^{1,2}(\Omega)$. So by the weak comparison principle, we have
$$\frac{1}{2}(1-t)u_{2}^{2}\leq u_{1},$$
this implies that
 $$e^{\frac{1}{2}(1-t)u_{2}^{2}}\in L^{1}(\Omega), $$
 so we have $e^{u_{2}}\in L^{q}$ for all $q\in(1,\infty)$. Next, consider
  \begin{equation}
\left\{
\begin{aligned}\nonumber
-Qu_{3}&=\lambda e^{u_{2}} \indent \mbox{in}\indent \Omega\\
u_{3}&=0 \indent \mbox{on}\indent\partial \Omega.\\
\end{aligned}
\right.
\end{equation}
we obtain $u_{3}\in L^{\infty}(\Omega)$, this means that $u_{3}$ is a regular supersolution of equation (\ref{3.2}), so, it follows from Theorem \ref{thm1.1}, the conclusion is obvious.
\end{proof}

It can be seen from Theorem \ref{thm1.1} that when $\lambda$ is small enough, the equation (\ref{1.3}) has solutions. In fact, when $N=2$, we can do it better. More precisely, we can use the Mountain Pass lemma to find at least two solutions. Since the energy functional corresponding to (\ref{1.3}) is
$$E(u)=\frac{1}{2}\int_{\Omega}F^{2}(\nabla u)dx-\lambda\int_{\Omega}e^{u}dx.$$

It follows from the Moser-Trudinger type inequality (\ref{+2.1}), by taking $\gamma=1$, we have
$$E(u)=\frac{1}{2}\int_{\Omega}F^{2}(\nabla u)dx-\lambda\int_{\Omega}e^{u}dx\geq\frac{1}{2}\int_{\Omega}F^{2}(\nabla u)dx-\lambda c_{1}|\Omega|e^{c_{2}\parallel F(\nabla u)\parallel^{2}_{L^{2}}},$$
it's easy to see that, there exist a constant $\lambda_{0}$ such that if $0<\lambda<\lambda_{0}$, the function
$$g(x)=\frac{1}{2}x^{2}-\lambda k_{1}|\Omega|e^{k_{2}\alpha^{2}x^{2}}$$
attain its maximum at $x_{0}\in(0,\infty)$ and $g(x_{0})>g(0)$. Therefore, there exist a point $u_{0}$ such that $E(u_{0})>E(0)$. Easy to check that $E(u)\rightarrow-\infty$ when $\parallel u\parallel_{W_{0}^{1,2}(\Omega)}$ large enough. That is to say, there is a mountain pass structure, the Mountain Pass lemma implies that there exist a Palais-Smale sequence $\{u_{j}\}$, thus, we only need to check $E$ verifies the  Palais-Smale condition.

\begin{lem}
The functional $E$ verifies the Palais-Smale condition.
\end{lem}

\begin{proof}
Let $\{u_{j}\}\subset W_{0}^{1,2}(\Omega)$ be a Palais-Smale sequence for $E$, i.e.
$$E(u_{j})\rightarrow c$$
and
$$E'(u_{j})\rightarrow0.$$

Writing $E'(u_{j})=\varepsilon_{j}$, where $\varepsilon_{j}\rightarrow0$ in $W^{-1, 2}(\Omega)$, we have
\begin{align}
c&=\lim_{j}\left\{E(u_{j})-\frac{1}{4}\langle\varepsilon_{j}, u_{j}\rangle+\frac{1}{4}\langle\varepsilon_{j}, u_{j}\rangle\right\}\nonumber\\
&\geq\lim_{j}\left\{\frac{1}{4}\int_{\Omega}|\nabla u_{j}|^{2}dx+\lambda\int_{\Omega}e^{u_{j}}(\frac{u_{j}}{4}-1)dx-\frac{1}{4}\left(\int_{\Omega}|\nabla u_{j}|^{2}dx\right)^{\frac{1}{2}}\right\}\nonumber\\
&\geq\lim_{j}\left\{\frac{1}{4}\int_{\Omega}|\nabla u_{j}|^{2}dx-\lambda C_{0}|\Omega|-\frac{1}{4}\left(\int_{\Omega}|\nabla u_{j}|^{2}dx\right)^{\frac{1}{2}}\right\},\nonumber
\end{align}
where $-C_{0}=\min_{x\in(0,\infty)}\{e^{x}(\frac{x}{4}-1)\}$. Here we obtain that $\{u_{j}\}$ is bounded in $W_{0}^{1,2}(\Omega)$. So there exist a subsequence, we still note by $j$, such that
$$u_{j}\rightarrow u\indent\mbox{weakly in}\indent W_{0}^{1,2}(\Omega)$$
and
$$u_{j}\rightarrow u\indent\mbox{strongly in}\indent L^{q}(\Omega),$$
for any $q\in(1,\infty)$. For any $\psi\in W_{0}^{1,2}(\Omega)$ with $\parallel\psi\parallel_{W_{0}^{1,2}(\Omega)}=1$, hence, by H\"{o}lder inequality, we have
\begin{align}
\left|\int_{\Omega}(e^{u_{j}}-e^{u})\psi dx\right|&\leq\int_{\Omega}e^{u}|e^{u_{j}-u}-1||\psi| dx\leq\int_{\Omega}e^{u}|u_{j}-u|e^{|u_{j}-u|}|\psi| dx\nonumber\\
&\leq\left(\int_{\Omega}e^{au}dx\right)^{\frac{1}{a}}\left(\int_{\Omega}e^{b|u_{j}-u|}dx\right)^{\frac{1}{b}}\left(\int_{\Omega}|u_{j}-u|^{c}dx\right)^{\frac{1}{c}}\left(\int_{\Omega}|\psi|^{d}dx\right)^{\frac{1}{d}}\rightarrow0,\nonumber
\end{align}
where $\frac{1}{a}+\frac{1}{b}+\frac{1}{c}+\frac{1}{d}=1$. Therefore, we have
$$\lim_{j}\left\{\sup_{\parallel\psi\parallel_{W_{0}^{1,2}(\Omega)}=1}\int_{\Omega}(e^{u_{j}}-e^{u})\psi dx\right\}=0,$$
this implies that $e^{u_{j}}\rightarrow e^{u}$ in $W^{-1,2}(\Omega)$. Since $E'(u_{j})\rightarrow0$, so for any $\psi\in W_{0}^{1,2}(\Omega)$, we have
$$E'(u_{j})[\psi]=\int_{\Omega}F(\nabla u_{j})F_{\xi}(\nabla u_{j})\nabla\psi-\lambda e^{u_{j}}\psi dx\rightarrow0,$$
this implies that
$$-Qu_{j}\rightarrow\lambda e^{u} \indent\mbox{in}\indent W^{-1,2}(\Omega),$$
by the continuity of the operator $(-Q)^{-1}:  W^{-1,2}(\Omega)\rightarrow W^{1,2}_{0}(\Omega)$, we obtain the result.
\end{proof}

It follows from the Mountain Pass lemma, the following result is obvious.
\begin{thm}
There exist a constant $\lambda_{0}$ such that if $0<\lambda<\lambda_{0}$, problem () admits a solution corresponding to a critical point of the functional $E$ with critical value
$$c=\inf_{\phi\in \mathfrak{C}}\max_{t\in[0,1]}E(\phi(t)),$$
where $\mathfrak{C}=\{\phi\in C([0,1], W_{0}^{1,2}(\Omega))|\phi(0)=0, \phi(1)=w_{0}\}$ for some $w_{0}\in W_{0}^{1,2}(\Omega)$ such that $E(w_{0})\leq E(0)$. Moreover $c>E(0)=-\lambda|\Omega|$.
\end{thm}

So, there exist $R_{2}>R_{1}>0$ such that $E(u_{1})=E(u_{2})\geq E(0)$ with $\parallel u_{1}\parallel_{W_{0}^{1,2}(\Omega)}=R_{1}$ and $\parallel u_{2}\parallel_{W_{0}^{1,2}(\Omega)}=R_{2}$. We consider the cut-off function $\varphi\in C_{c}^{\infty}(\mathbb{R})$ which satisfies $\varphi(x)=1$ for $x\leq R_{1}$, $\varphi(x)=0$ if $x\geq R_{2}$ and $\varphi(x)$ is nonincreasing. We define the functional
$$J(u)=\frac{1}{2}\int_{\Omega}F^{2}(\nabla u)dx-\lambda\int_{\Omega}\varphi(\parallel u\parallel_{W_{0}^{1,2}(\Omega)})e^{u}dx.$$

\begin{lem}
$\inf_{u\in W_{0}^{1,2}(\Omega)}J(u)<-\lambda|\Omega|.$
\end{lem}

\begin{proof}
Take $v_{0}\in W_{0}^{1,2}(\Omega)$, $v_{0}\geq0$ and $\parallel v_{0}\parallel_{W_{0}^{1,2}(\Omega)}=1$. For $\rho\leq R_{1}$, we have
\begin{align}
J(\rho v_{0})&=\frac{1}{2}\rho^{2}-\lambda\int_{\Omega}e^{\rho v_{0}}dx\leq \frac{1}{2}\rho^{2}-\lambda|\Omega|-\lambda\rho\int_{\Omega}v_{0}dx\nonumber\\
&=\rho\left(\frac{1}{2}\rho-\lambda\int_{\Omega}v_{0}dx\right)-\lambda|\Omega|,\nonumber
\end{align}
choose $\rho$ small enough, we obtain the result.
\end{proof}
Therefore, for $\parallel u\parallel_{W_{0}^{1,2}(\Omega)}\leq R_{1}$, there exist a Palais-Smale sequence, meanwhile, in this domain $J(u)=E(u)$, thus the functional $J$ also satisfies the Palais-Smale condition.

\begin{thm}
There exist a constant $\lambda_{0}>0$ such that the functional $E$ has a critical point  with critical value $c'<-\lambda|\Omega|$, if $\lambda\in(0,\lambda_{0})$.
\end{thm}

Next, we prove the nonexistence result Theorem \ref{thm1.2}.

\begin{proof}[Proof of Theorem \ref{thm1.2}]
Since $\lambda>\lambda_{1}$, so there exist $\delta>0$ small enough such that $\lambda\geq\lambda_{1}+\delta$. Let $\lambda_{\delta}=\lambda_{1}+\delta$ and $v_{1}$ be a positive eigenfunction associate with $\lambda_{1}$ with $\parallel v_{1}\parallel_{L^{\infty}}\leq1$. Suppose $u\in W_{0}^{1,2}(\Omega)$ be a weak solution of equation (\ref{1.3}), hence, we have
$$-Qv_{1}\leq\lambda_{\delta}v_{1}\leq\lambda_{\delta}\leq\lambda e^{u}=-Qu,$$
by the weak comparision principle, we have $v_{1}\leq u$. Let $v_{2}$ be the solution of
\begin{equation}
\left\{
\begin{aligned}\nonumber
-Qv_{2}&=\lambda_{\delta} v_{1} \indent \mbox{in}\indent\Omega\\
v_{2}&=0 \indent \mbox{on}\indent\partial\Omega,\\
\end{aligned}
\right.
\end{equation}
easy to check $v_{2}\in C^{1,\alpha}(\Omega)$, moreover
$$-Qv_{2}=\lambda_{\delta} v_{1}\leq\lambda e^{u}=-Qu$$
and
$$-Qv_{1}\leq\lambda_{\delta}v_{1}=-Qv_{2},$$
by the weak comparision principle, we have
$$v_{1}\leq v_{2}\leq u.$$
By the similar way, we can obtain a increasing sequence $\{v_{n}\}$ which has a upper bounded $u$, hence, passing to the limit, we get a function $v_{0}\in W_{0}^{1,2}(\Omega)$ which solves
\begin{equation}
\left\{
\begin{aligned}\nonumber
-Qv_{0}&=\lambda_{\delta} v_{0} \indent \mbox{in}\indent\Omega\\
v_{0}&=0 \indent \mbox{on}\indent\partial\Omega,\\
\end{aligned}
\right.
\end{equation}
This is impossible for $\delta$ small enough, because Proposition \ref{+pro2.5} has proved that the first eigenvalue is simple.
\end{proof}

\begin{proof}[Proof of Theorem \ref{thm1.3}]
Since $\underline{u}_{n}$ is the minimal solution, by the properties of $F$, we have
\begin{align}
\lambda_{n}\int_{\Omega}e^{\underline{u}_{n}}\underline{u}_{n}^{2}dx&\leq\int_{\Omega}F_{\xi_{i}}(\nabla\underline{u}_{n})F_{\xi_{j}}(\nabla\underline{u}_{n})\underline{u}_{nx_{i}}\underline{u}_{nx_{j}}+F(\nabla\underline{u}_{n})F_{\xi_{i}\xi_{j}}(\nabla\underline{u}_{n})\underline{u}_{nx_{i}}\underline{u}_{nx_{j}}dx\nonumber\\
&=\int_{\Omega}F^{2}(\nabla\underline{u}_{n})dx=\lambda_{n}\int_{\Omega}e^{\underline{u}_{n}}\underline{u}_{n}dx.\nonumber
\end{align}
Consider $T_{n}=\{x\in\Omega|\underline{u}_{n}>2\}$, then, in $\Omega\setminus T_{n}$, we have $0\leq \underline{u}_{n}\leq2$, it follows that
\begin{align}
\int_{\Omega}e^{\underline{u}_{n}}\underline{u}_{n}^{2}dx&\leq\int_{\Omega}e^{\underline{u}_{n}}\underline{u}_{n}dx=\int_{\Omega\setminus T_{n}}e^{\underline{u}_{n}}\underline{u}_{n}dx+\int_{T_{n}}e^{\underline{u}_{n}}\underline{u}_{n}dx\nonumber\\
&\leq2e^{2}|\Omega|+\frac{1}{2}\int_{\Omega}e^{\underline{u}_{n}}\underline{u}_{n}^{2}dx,\nonumber
\end{align}
hence, we have
$$\int_{\Omega}e^{\underline{u}_{n}}\underline{u}_{n}^{2}dx\leq4e^{2}|\Omega|.$$
Since
\begin{align}
\int_{\Omega}e^{\underline{u}_{n}}\underline{u}_{n}dx&=\int_{\Omega\setminus T_{n}}e^{\underline{u}_{n}}\underline{u}_{n}dx+\int_{T_{n}}e^{\underline{u}_{n}}\underline{u}_{n}dx\nonumber\\
&\leq2e^{2}|\Omega|+\frac{1}{2}\int_{\Omega}e^{\underline{u}_{n}}\underline{u}_{n}^{2}dx\nonumber\\
&\leq4e^{2}|\Omega|,\nonumber
\end{align}
and
$$\int_{\Omega}F^{2}(\nabla\underline{u}_{n})dx=\lambda_{n}\int_{\Omega}e^{\underline{u}_{n}}\underline{u}_{n}dx\leq\lambda^{*}4e^{2}|\Omega|,$$
therefore, for a subsequence, we still note by $\{\underline{u}_{n}\}$, we have
$$\underline{u}_{n}\rightarrow u^{*}\indent\mbox{weakly in}\indent W_{0}^{1,2}(\Omega),$$
it follows from the Monotone convergence that
$$e^{\underline{u}_{n}}\rightarrow e^{u^{*}}\indent\mbox{strongly in} \indent L^{1}(\Omega).$$
Since, for any $\phi\in W_{0}^{1,2}(\Omega)$ and $\psi\in W_{0}^{1,2}(\Omega)$, we have
$$\int_{\Omega}F(\nabla \underline{u}_{n})F_{\xi}(\nabla \underline{u}_{n})\nabla\phi dx=\lambda_{n}\int_{\Omega}e^{\underline{u}_{n}}\phi dx,$$
and
$$\int_{\Omega}F_{\xi_{i}}(\nabla \underline{u})F_{\xi_{j}}(\nabla \underline{u})\psi_{x_{i}}\psi_{x_{j}}+F(\nabla \underline{u})F_{\xi_{i}\xi_{j}}(\nabla \underline{u})\psi_{x_{i}}\psi_{x_{j}}dx\geq\lambda_{n}\int_{\Omega}e^{\underline{u}}\psi^{2}dx,$$
for $\alpha\in(0,2)$, take $\phi=\frac{1}{2\alpha}(e^{2\alpha\underline{u}_{n}}-1)$ and $\psi=e^{\alpha\underline{u}_{n}}-1$, we have
$$\lambda_{n}\int_{\Omega}e^{\underline{u}_{n}}\frac{1}{2\alpha}(e^{2\alpha\underline{u}_{n}}-1) dx=\int_{\Omega}F(\nabla \underline{u}_{n})F_{\xi}(\nabla \underline{u}_{n})\nabla\phi dx=\int_{\Omega}F^{2}(\nabla \underline{u}_{n})e^{2\alpha \underline{u}_{n}}dx,$$
and
\begin{align}
\lambda_{n}\int_{\Omega}e^{\underline{u}_{n}}(e^{\alpha\underline{u}_{n}}-1)^{2}dx&\leq\int_{\Omega}F_{\xi_{i}}(\nabla\underline{u}_{n})F_{\xi_{j}}(\nabla\underline{u}_{n})\alpha^{2}e^{2\alpha\underline{u}_{n}}\underline{u}_{nx_{i}}\underline{u}_{nx_{j}}dx\nonumber\\
&\indent+F(\nabla\underline{u}_{n})F_{\xi_{i}\xi_{j}}(\nabla\underline{u}_{n})\alpha^{2}e^{2\alpha\underline{u}_{n}}\underline{u}_{nx_{i}}\underline{u}_{nx_{j}}dx\nonumber\\
&=\int_{\Omega}F^{2}(\nabla\underline{u}_{n})dx=\lambda_{n}\int_{\Omega}e^{\underline{u}_{n}}\underline{u}_{n}dx.\nonumber
\end{align}
It follows that
$$\frac{1}{\alpha}\int_{\Omega}e^{\underline{u}_{n}}(e^{\alpha\underline{u}_{n}}-1)^{2}dx\leq\frac{1}{2}\int_{\Omega}e^{\underline{u}_{n}}(e^{2\alpha\underline{u}_{n}}-1) dx,$$
since $\alpha<2$, by Young's inequality, we have
$$C(\alpha)\int_{\Omega}e^{(2\alpha+1)\underline{u}_{n}}dx\leq\int_{\Omega}e^{\underline{u}_{n}}dx\leq\int_{\Omega}e^{u^{*}}dx.$$
Since, we can choose $\alpha<2$ such that $2\alpha+1\geq\frac{2^{*}}{2^{*}-1}$, moreover $L^{\frac{2^{*}}{2^{*}-1}}(\Omega)\subset W^{-1,2}(\Omega)$, hence, we have
$$e^{\underline{u}_{n}}\rightarrow e^{u^{*}}\indent\mbox{in}\indent W^{-1,2}(\Omega).$$
By the continuity of $(-Q)^{-1}$, we have
$$\underline{u}_{n}\rightarrow u^{*}\indent\mbox{strongly in }\indent W_{0}^{1,2}(\Omega),$$
hence, for any $\phi\in W^{-1,2}(\Omega)$, we have
\begin{align}
\int_{\Omega}F(\nabla u^{*})F_{\xi}(\nabla u^{*})\nabla\phi dx&=\lim_{n}\int_{\Omega}F(\nabla \underline{u}_{n})F_{\xi}(\nabla \underline{u}_{n})\nabla\phi dx\nonumber\\
&=\lim_{n}\int_{\Omega}e^{\underline{u}_{n}}\phi dx=\int_{\Omega}e^{u^{*}}\phi dx.\nonumber
\end{align}

\end{proof}

At the end of this section, we give the proof of the regularity result of extremal solution.

\begin{proof}
Similar to the above argument, we have
$$\frac{1}{\alpha}\int_{\Omega}e^{u}(e^{\alpha u}-1)^{2}dx\leq\frac{1}{2}\int_{\Omega}e^{u}(e^{2\alpha u}-1) dx,$$
for some parameter $\alpha$. If $\alpha<2$, by Young's inequality, we have
$$\int_{\Omega}e^{(2\alpha+1)u}dx\leq C,$$
if $2\alpha+1>\frac{N}{2}$ i.e. $\alpha>\frac{N-2}{4}$, we have $N<10$, then Lemma \ref{+lem2.1} implies that $u\in L^{\infty}(\Omega)$.
\end{proof}

\section{Some Liouville theorems}\label{sec4}
In this section, we give the proof of Liouville theorems. Before proving our main results, we first explain why we need $\alpha>-2$, see the following lemma.
\begin{lem}\label{lem2.1}
For $\alpha\leq-2$, under the assumption (\ref{1.6}), there is no weak solution for equation
$$-Qu=(F^{0}(x))^{\alpha}e^{u} \indent\mbox{in}\indent\Omega,$$
where $\Omega\subset\mathbb{R}^{N}$ (possibly unbounded) containg $0$.
\end{lem}
\begin{proof}
Let $B_{R}(0)\subset\Omega$ and $v(r)=\fint_{\partial B_{r}}udS$ for $r\in(0,R)$, we have
$$v'(r)=\frac{1}{N\kappa_{0}r^{N-1}}\int_{\partial B_{r}}\langle \nabla u,\frac{x}{r}\rangle dS.$$
By the assumption of (\ref{1.6}), we have
$$\langle \nabla u,x\rangle=F(\nabla u)\langle F_{\xi}(\nabla u), F^{0}_{\xi}(x)\rangle F^{0}(x),$$
and $F^{0}(x)=r$, $\nu=F^{0}_{\xi}(x)$ on $\partial B_{r}$. So it follows that
\begin{align}
-v'(r)&=\frac{1}{N\kappa_{0}r^{N-1}}\int_{\partial B_{r}}-\sum_{i=1}^{N}F(\nabla u)F_{\xi_{i}}(\nabla u)\nu_{i}dS=\frac{1}{N\kappa_{0}r^{N-1}}\int_{B_{r}}-Qudx\nonumber\\
&=\frac{1}{N\kappa_{0}r^{N-1}}\int_{B_{r}}(F^{0}(x))^{\alpha}e^{u}dx=\frac{1}{N\kappa_{0}r^{N-1}}\int_{0}^{r}t^{\alpha}\int_{\partial B_{t}}e^{u}dSdt\nonumber\\
&=\frac{1}{r^{N-1}}\int_{0}^{r}t^{N-1+\alpha}\fint_{\partial B_{t}}e^{u}dSdt=\frac{1}{r^{N-1}}\int_{0}^{r}t^{N-1+\alpha}e^{v(t)}dt\nonumber\\
&\geq\frac{1}{N+\alpha}r^{1+\alpha}e^{v(r)},\nonumber
\end{align}
we deduce that
$$e^{-v(r)}\geq C\int_{r_{1}}^{r}t^{1+\alpha}dt\rightarrow\infty,$$
if $r_{1}\rightarrow0$ and $\alpha\leq-2$. This is a contradiction.
\end{proof}

In order to prove our main results, we use the Moser iteration argument to prove the following elliptic estimate, which is inspired by \cite{AY,F,FL,WY} and references therein.
\begin{pro}\label{pro2.1}
Let $u$ be a weak solution of
\begin{align}\label{2.1}
-Qu=(F^{0}(x))^{\alpha}e^{u} \indent\mbox{in}\indent\Omega,
\end{align}
which is stable, that is satisfies
\begin{align}\label{2.2}
\int_{\Omega}F_{\xi_{i}}(\nabla u)F_{\xi_{j}}(\nabla u)\phi_{x_{i}}\phi_{x_{j}}+F(\nabla u)F_{\xi_{i}\xi_{j}}(\nabla u)\phi_{x_{i}}\phi_{x_{j}}-(F^{0}(x))^{\alpha}e^{u}\phi^{2}dx\geq0,
\end{align}
for all $\phi\in C_{c}^{\infty}(\Omega)$, here $N\geq2$ and $\Omega\subset\mathbb{R}^{N}$ (possibly unbounded). Then for any $\beta\in(0,4)$ and integer $m\geq10$, we have
\begin{align}\label{2.3}
\int_{\Omega}(F^{0}(x))^{\alpha}e^{(\beta+1)u}\psi^{2m}dx\leq C\int_{\Omega}(F^{0}(x))^{-\beta\alpha}\left(|\nabla\psi|^{2}+|\nabla\psi|^{4}\right)^{\beta+1}dx ,
\end{align}
here, $\psi\in C_{c}^{1}(\Omega)$ is a test function and satisfying $0\leq\psi\leq1$ in $\Omega$.

\end{pro}

\begin{proof}
For any $\beta\in(0,4)$ and any $k>0$, we set
$$ a_{k}(t)=\left\{
\begin{aligned}
e^{\frac{\beta t}{2}}, &  & \indent\mbox{if}\indent t<k \\
[\frac{\beta}{2}(t-k)+1]e^{\frac{\beta k}{2}}, & & \indent\mbox{if}\indent t\geq k, \\
\end{aligned}
\right.
$$
and
$$ b_{k}(t)=\left\{
\begin{aligned}
e^{\beta t}, &  & \indent\mbox{if}\indent t<k \\
[\beta(t-k)+1]e^{\beta k}, & & \indent\mbox{if}\indent t\geq k. \\
\end{aligned}
\right.
$$
Simple calculations yields
\begin{align}\label{2.4}
a_{k}^{2}(t)\geq b_{k}(t),\indent (a'_{k}(t))^{2}=\frac{\beta}{4}b_{k}'(t) ,
\end{align}
and
\begin{align}\label{2.5}
(a_{k}'(t))^{-2}(a_{k}(t))^{4}\leq c_{1}e^{\beta t},\indent (a_{k}(t))^{2}\leq e^{\beta t}, \indent (b_{k}'(t))^{-1}(b_{k}(t))^{2}\leq c_{2}e^{\beta t},
\end{align}
for some positive constant $c_{1}$ and $c_{2}$ which depends only on $\beta$. For any $\phi\in C_{c}^{1}(\Omega)$, take $b_{k}(u)\phi^{2}$ as the test function, multiply (\ref{2.1}) and integrate by parts, it follows from Proposition \ref{pro1.1} , we have
\begin{align}
&\int_{\Omega}-div(F(\nabla u)F_{\xi}(\nabla u))b_{k}(u)\phi^{2}dx\nonumber\\
&=\int_{\Omega}-\frac{\partial}{\partial x_{i}}(F(\nabla u)F_{\xi_{i}}(\nabla u))b_{k}(u)\phi^{2}dx\nonumber\\
&=\int_{\Omega}F(\nabla u)F_{\xi_{i}}(\nabla u)b_{k}'(u)u_{x_{i}}\phi^{2}+F(\nabla u)F_{\xi_{i}}(\nabla u)b_{k}(u)2\phi\phi_{x_{i}}dx\nonumber\\
&=\int_{\Omega}F^{2}(\nabla u)b_{k}'(u)\phi^{2}+F(\nabla u)F_{\xi_{i}}(\nabla u)b_{k}(u)2\phi\phi_{x_{i}}dx\nonumber\\
&=\int_{\Omega}(F^{0}(x))^{\alpha}e^{u}b_{k}(u)\phi^{2}dx . \nonumber
\end{align}
Since $|F_{\xi}(\nabla u)|\leq C$, it follows that
\begin{align}
\int_{\Omega}F^{2}(\nabla u)b_{k}'(u)\phi^{2}dx \leq 2C\int_{\Omega}F(\nabla u)b_{k}(u)|\phi||\nabla\phi|dx+\int_{\Omega}(F^{0}(x))^{\alpha}e^{u}b_{k}(u)\phi^{2}dx,\nonumber
\end{align}
by the Cauchy inequality, we have
 \begin{align}\label{2.6}
\int_{\Omega}F^{2}(\nabla u)b_{k}'(u)\phi^{2}dx \leq & \frac{2C}{(1-2C\varepsilon)\varepsilon}\int_{\Omega}(b_{k}'(u))^{-1}b_{k}^{2}(u)|\nabla\phi|^{2}dx\nonumber\\
&+\frac{1}{1-2C\varepsilon}\int_{\Omega}(F^{0}(x))^{\alpha}e^{u}b_{k}(u)\phi^{2}dx.
\end{align}

Take $\varphi=a_{k}(u)\phi$ as the test function in (\ref{2.2}), using Proposition \ref{pro1.1} and Cauchy inequality, we have
\begin{align}
&\int_{\Omega}F_{\xi_{i}}(\nabla u)F_{\xi_{j}}(\nabla u)\varphi_{x_{i}}\varphi_{x_{j}}dx\nonumber\\
&=\int_{\Omega}F_{\xi_{i}}(\nabla u)F_{\xi_{j}}(\nabla u)(a_{k}'(u)u_{x_{i}}\phi+a_{k}(u)\phi_{x_{i}})(a_{k}'(u)u_{x_{j}}\phi+a_{k}(u)\phi_{x_{j}})dx\nonumber\\
&\leq\int_{\Omega}(1+2C\varepsilon_{1})F^{2}(\nabla u)(a_{k}'(u))^{2}\phi^{2}+(C^{2}+\frac{2C}{\varepsilon_{1}})(a_{k}(u))^{2}|\nabla\phi|^{2}dx,\nonumber
\end{align}
and it follows from (\ref{1.2}), we have
\begin{align}
&\int_{\Omega}F(\nabla u)F_{\xi_{i}\xi_{j}}(\nabla u)\varphi_{x_{i}}\varphi_{x_{j}}dx\nonumber\\
&=\int_{\Omega}F(\nabla u)F_{\xi_{i}\xi_{j}}(\nabla u)(a_{k}'(u)u_{x_{i}}\phi+a_{k}(u)\phi_{x_{i}})(a_{k}'(u)u_{x_{j}}\phi+a_{k}(u)\phi_{x_{j}})dx\nonumber\\
&=\int_{\Omega}F(\nabla u)F_{\xi_{i}\xi_{j}}(a_{k}(u))^{2}\phi_{x_{i}}\phi_{x_{j}}dx\leq\Lambda\int_{\Omega}F(\nabla u)(a_{k}(u))^{2}|\nabla\phi|^{2}dx\nonumber\\
&\leq\Lambda\varepsilon_{2}\int_{\Omega}F^{2}(\nabla u)(a_{k}'(u))^{2}\phi^{2}dx+\frac{\Lambda}{\varepsilon_{2}}\int_{\Omega}(a_{k}'(u))^{-2}(a_{k}(u))^{4}\frac{|\nabla\phi|^{4}}{\phi^{2}}dx.\nonumber
\end{align}
Hence,
\begin{align}\label{2.7}
\int_{\Omega}(F^{0}(x))^{\alpha}e^{u}(a_{k}(u))^{2}\phi^{2}dx\leq&\int_{\Omega}(1+2C\varepsilon_{1}+\Lambda\varepsilon_{2})F^{2}(\nabla u)(a_{k}'(u))^{2}\phi^{2}dx\nonumber\\
&+(C^{2}+\frac{2C}{\varepsilon_{1}})(a_{k}(u))^{2}|\nabla\phi|^{2}dx\nonumber\\
&+\frac{\Lambda}{\varepsilon_{2}}\int_{\Omega}(a_{k}'(u))^{-2}(a_{k}(u))^{4}\frac{|\nabla\phi|^{4}}{\phi^{2}}dx.
\end{align}
Combine (\ref{2.4}), (\ref{2.5}), (\ref{2.6}) and (\ref{2.7}),  we obtain
\begin{align}
\int_{\Omega}(F^{0}(x))^{\alpha}e^{u}(a_{k}(u))^{2}\phi^{2}dx
&\leq\frac{\beta(1+2C\varepsilon_{1}+\Lambda\varepsilon_{2})}{4(1-2C\varepsilon)}\int_{\Omega}(F^{0}(x))^{\alpha}e^{u}(a_{k}(u))^{2}\phi^{2}dx\nonumber\\
&\indent+\left[\frac{2C\beta(1+2C\varepsilon_{1}+\Lambda\varepsilon_{2})}{4(1-2C\varepsilon)\varepsilon}c_{2}+C^{2}+\frac{2C}{\varepsilon_{1}}\right]\int_{\Omega}e^{\beta u}|\nabla\phi|^{2}dx \nonumber\\
&\indent+\frac{\Lambda}{\varepsilon_{2}}c_{1}\int_{\Omega}e^{\beta u}\frac{|\nabla\phi|^{4}}{\phi^{2}}dx.\nonumber
\end{align}
Since $\beta\in(0,4)$, so we can choose $\varepsilon$, $\varepsilon_{1}$ and $\varepsilon_{2}$ small enough, such that
$$\frac{\beta(1+2C\varepsilon_{1}+\Lambda\varepsilon_{2})}{4(1-2C\varepsilon)}<1 . $$
Hence, we have
$$\int_{\Omega}(F^{0}(x))^{\alpha}e^{u}(a_{k}(u))^{2}\phi^{2}dx\leq C_{1}\int_{\Omega}e^{\beta u}\frac{|\nabla\phi|^{4}}{\phi^{2}}dx+C_{2}\int_{\Omega}e^{\beta u}|\nabla\phi|^{2}dx,$$
for some positive constants $C_{1}$ and $C_{2}$ which are independent of $k$. Then let $k\rightarrow+\infty$, Fatou's lemma tell us that
$$\int_{\Omega}(F^{0}(x))^{\alpha}e^{(\beta+1)u}\phi^{2}dx\leq C_{1}\int_{\Omega}e^{\beta u}\frac{|\nabla\phi|^{4}}{\phi^{2}}dx+C_{2}\int_{\Omega}e^{\beta u}|\nabla\phi|^{2}dx . $$
Let $\phi=\psi^{m}$ and $0\leq\psi\leq1$, by young's inequality, we have
\begin{align}
&\int_{\Omega}(F^{0}(x))^{\alpha}e^{(\beta+1)u}\phi^{2}dx=\int_{\Omega}(F^{0}(x))^{\alpha}e^{(\beta+1)u}\psi^{2m}dx\nonumber\\
&\leq \widetilde{C_{1}}\varepsilon\int_{\Omega}(F^{0}(x))^{\alpha}e^{(\beta+1)u}\psi^{2m}dx+\frac{\widetilde{C_{1}}}{\varepsilon}\int_{\Omega}(F^{0}(x))^{-\beta\alpha}\left(|\psi|^{2m-2-2m\frac{\beta}{\beta+1}}|\nabla\psi|^{2}\right)^{\beta+1}dx\nonumber\\
&\indent+\widetilde{C_{2}}\varepsilon\int_{\Omega}(F^{0}(x))^{\alpha}e^{(\beta+1)u}\psi^{2m}dx+\frac{\widetilde{C_{2}}}{\varepsilon}\int_{\Omega}(F^{0}(x))^{-\beta\alpha}\left(|\psi|^{2m-4-2m\frac{\beta}{\beta+1}}|\nabla\psi|^{4}\right)^{\beta+1}dx . \nonumber
\end{align}
Since $m\geq10$, we have $2m-4-2m\frac{\beta}{\beta+1}\geq0$ and we can choose $\varepsilon$ small such that
$$\int_{\Omega}(F^{0}(x))^{\alpha}e^{(\beta+1)u}\psi^{2m}dx\leq C\int_{\Omega}(F^{0}(x))^{-\beta\alpha}(|\nabla\psi|^{2}+|\nabla\psi|^{4})^{\beta+1}dx . $$
This completes the proof of (\ref{2.3}).
\end{proof}

From the above elliptic estimate, we can start to prove our main theorems.
\begin{proof}[Proof of Theorem \ref{thm1.5}]
By contradiction, suppose $u$ is a stable solution of equation, since $N< 10+4\alpha$, fix $m\geq10$, choose $\beta\in(0,4)$ such that $N-\beta\alpha-2(\beta+1)<0$. Then, for every $x\in\mathbb{R}^{N}$, consider the function $\phi_{R}(x)=\phi\left(\frac{F^{0}(x)}{R}\right)$ where $\phi\in C_{c}^{1}(\mathbb{R})$ satisfies $0\leq\phi\leq1$,
$$ \phi(t)=\left\{
\begin{aligned}
1, &  & \indent\mbox{if}\indent |t|<1 \\
0, & & \indent\mbox{if}\indent |t|\geq 2. \\
\end{aligned}
\right.
$$
we have
$$\int_{B_{R}(0)}(F^{0}(x))^{\alpha}e^{(\beta+1)u}dx\leq CR^{N-\beta\alpha-2(\beta+1)},$$
let $R\rightarrow\infty$, we obtain $\int_{\mathbb{R}^{N}}(F^{0}(x))^{\alpha}e^{(\beta+1)u}dx=0$, a contradiction.
\end{proof}

Next, we will prove the main Theorem \ref{thm1.5} by contradiction. Suppose there exist a compact set $S$, such that $u$ is a weak stable solution of (\ref{1.5}) in $\mathbb{R}^{N}\setminus S$. We can choose $R_{0}>0$ such that $S\subset B_{R_{0}}$. Therefore we can apply Proposition \ref{pro2.1} with $\Omega=\mathbb{R}^{N}\setminus B_{R_{0}}$. In order to derive the contradiction, we need the following estimates.

\begin{lem}
For any $\beta\in(0,4)$ and $r>R_{0}+3$, there exist positive constant $A$ and $B$ independent of $r$, holds

\begin{align}\label{2.8}
\int_{B_{r}\setminus B_{R_{0}+2}}(F^{0}(x))^{\alpha}e^{(\beta+1)u}dx\leq A+Br^{N-2(\beta+1)-\beta\alpha}.
\end{align}
Moreover, for any $B_{2R}(y)\subset\{x\in\mathbb{R}^{N}: F^{0}(x)>R_{0}\}$, we have
\begin{align}\label{2.9}
\int_{B_{2R}(y)}(F^{0}(x))^{\alpha}e^{(\beta+1)u}dx\leq CR^{N-2(\beta+1)-\beta\alpha},
\end{align}
where $C$ is a positive constant independent of $R$ and $y$.
\end{lem}

\begin{proof}
Fix $m=10$, and for every $r>R_{0}+3$, we consider the following test function $\xi_{r}\in C_{c}^{1}(\mathbb{R}^{N})$
$$ \xi_{r}(x)=\left\{
\begin{aligned}
\theta_{R_{0}}(F^{0}(x)), &  & \indent\mbox{if}\indent x\in B_{R_{0}+3} \\
\phi\left(\frac{F^{0}(x)}{r}\right), &  &\indent\mbox{if}\indent x\in \mathbb{R}^{N}\setminus B_{R_{0}+3}, \\
\end{aligned}
\right.
$$
where $\phi$ is defined in the Proof of Theorem \ref{thm1.1} and for $s>0$, $\theta_{s}$ satisfying $\theta_{s}\in C_{c}^{1}(\mathbb{R})$, $0\leq\theta_{s}\leq1$ everywhere on $\mathbb{R}$ and
$$ \theta_{s}(t)=\left\{
\begin{aligned}
0, &  & \indent\mbox{if}\indent |t|\leq s+1 \\
1, &  &\indent\mbox{if}\indent |t|\geq s+2. \\
\end{aligned}
\right.
$$
It follows from Proposition \ref{pro2.1} that
\begin{align}
\int_{B_{r}\setminus B_{R_{0}+2}}(F^{0}(x))^{\alpha}e^{(\beta+1)u}dx&\leq\int_{\Omega}(F^{0}(x))^{\alpha}e^{(\beta+1)u}dx\nonumber\\
&\leq C\int_{\Omega}(F^{0}(x))^{-\beta\alpha}\left(|\nabla\xi_{r}|^{2}+|\nabla\xi_{r}|^{4}\right)^{\beta+1}dx\nonumber\\
&\leq C_{1}(\alpha,N,\theta_{R_{0}})+C_{2}(\alpha,N,\phi)r^{N-2(\beta+1)-\beta\alpha},\nonumber
\end{align}
hence the inequality (\ref{2.8}) holds.

The integral estimate (\ref{2.9}) is obtained in the same way by using the test functions $\psi_{R,y}(x)=\phi(\frac{F^{0}(x-y)}{R})$ in Proposition \ref{pro2.1}.
\end{proof}

\begin{lem}
$$\lim_{F^{0}(x)\rightarrow\infty}(F^{0}(x))^{2+\alpha}e^{u(x)}=0.$$
\end{lem}
\begin{proof}
Consider $g(\beta)=N-2(\beta+1)-\beta\alpha^{-}$, since $3\leq N<10+4\alpha^{-}$, we have $g(0)>0$ and $g(4)<0$, thus, there exist $\beta_{1}\in(0,4)$ and $\varepsilon_{0}\in(0,2)$ such that
$$\beta_{1}+1\geq\beta_{1}+1+\frac{\beta_{1}}{2}\alpha^{-}>\theta:=\frac{N}{2-\varepsilon_{0}}>\frac{N}{2}.$$
Let $F^{0}(y)>4R_{0}$, $R=\frac{F^{0}(y)}{4}$, we have $B_{2R}(y)\subset\mathbb{R}^{N}\setminus\overline{B_{R_{0}}}$. By H\"{o}lder inequality
\begin{align}
\int_{B_{R}(y)}\left((F^{0}(x))^{\alpha}e^{u}\right)^{\theta}dx&\leq\left(\int_{B_{R}(y)}(F^{0}(x))^{\alpha}e^{(\beta_{1}+1)u}dx\right)^{\frac{\theta}{\beta_{1}+1}}\left(\int_{B_{R}(y)}(F^{0}(x))^{\frac{\beta_{1}\alpha\theta}{\beta_{1}+1-\theta}}\right)^{\frac{\beta_{1}+1-\theta}{\beta_{1}+1}}\nonumber\\
&\leq C\left(R^{N-2(\beta_{1}+1)-\beta_{1}\alpha}\right)^{\frac{\theta}{\beta_{1}+1}}\left(R^{N+\frac{\beta_{1}\alpha\theta}{\beta_{1}+1-\theta}}\right)^{\frac{\beta_{1}+1-\theta}{\beta_{1}+1}}\nonumber\\
&=CR^{N-2\theta}.\nonumber
\end{align}

Let $\beta_{2}=\frac{N-2}{2+\alpha}$ and $\lambda=\frac{N+\alpha}{2(2+\alpha)}$, then $\beta_{2}\in(0,4)$. Since $3\leq N<10+4\alpha^{-}$ and $N-2(\beta_{2}+1)-\beta_{2}\alpha=0$, let $w=e^{\lambda u}$, from (\ref{2.8}), take $\beta=\beta_{2}$ and let $r\rightarrow\infty$, we have
$$\int_{\mathbb{R}^{N}\setminus B_{R_{0}+2}}(F^{0}(x))^{\alpha}e^{(\beta+1)u}dx=\int_{\mathbb{R}^{N}\setminus B_{R_{0}+2}}(F^{0}(x))^{\alpha}w^{2}dx<\infty.$$

Since $-Qw-\lambda(F^{0}(x))^{\alpha}e^{u}w=-\lambda^{2}e^{\lambda u}F^{2}(\nabla u)\leq0$, by the Harnack inequality of Serrin \cite{S}, we have, for any $q>1$
$$\parallel w\parallel_{L^{\infty}(B_{R}(y))}\leq CR^{-\frac{N}{q}}\parallel w\parallel_{L^{q}(B_{2R}(y))},$$
where $C$ is a positive constant dependent on $N$ and $R^{\varepsilon_{0}}\parallel \lambda(F^{0}(x))^{\alpha}e^{u}\parallel_{L^{\frac{N}{2-\varepsilon_{0}}}(B_{2R}(y))}$. Let $\theta=\frac{N}{2-\varepsilon_{0}}$, we have
$$R^{\varepsilon_{0}}\parallel\lambda(F^{0}(x))^{\alpha}e^{u}\parallel_{L^{\theta}(B_{2R}(y))}\leq C\lambda R^{\varepsilon_{0}}(R^{N-2\theta})^{\frac{1}{\theta}}=C\lambda.$$
Let $q=2$, we have
$$w(y)\leq CR^{-\frac{N}{2}}\parallel w\parallel_{L^{2}}\leq CR^{-\frac{N}{2}}R^{\frac{-\alpha}{2}}\parallel (F^{0}(x))^{\frac{\alpha}{2}}w\parallel_{L^{2}}=o(R^{-\frac{N+\alpha}{2}}),$$
since $e^{u(y)}=w(y)^{\frac{1}{\lambda}}$, it follows that
$$\lim_{F^{0}(x)\rightarrow\infty}(F^{0}(x))^{2+\alpha}e^{u(x)}=0.$$
\end{proof}

Now, it's suffice to finish the proof of Theorem \ref{thm1.6}.
\begin{proof}[Proof of Theorem \ref{thm1.6}]
Let $v(r)=\frac{1}{N\kappa_{0}r^{N-1}}\int_{\partial B_{r}}udS$, similar with Lemma \ref{lem2.1}, we deduce
\begin{align}
v'(r)=\frac{1}{N\kappa_{0}r^{N-1}}\int_{\partial B_{r}}\sum_{i=1}^{N}F(\nabla u)F_{\xi_{i}}(\nabla u)\nu_{i} dS=\frac{1}{N\kappa_{0}r^{N-1}}\int_{B_{r}}Qudx . \nonumber
\end{align}
Since $\alpha>-2$, so there exist $\delta>0$ such that $2+\alpha-\delta>0$. It follows from that $\lim_{F^{0}(x)\rightarrow\infty}(F^{0}(x))^{2+\alpha}e^{u(x)}=0$, then, for $r$ large enough, we have
$$-v'(r)=\frac{1}{N\kappa_{0}r^{N-1}}\int_{B_{r}}-Qudx=\frac{1}{N\kappa_{0}r^{N-1}}\int_{B_{r}}(F^{0}(x))^{\alpha}e^{u}dx\leq\frac{\delta}{r},$$
to go further, we have
$$r^{2+\alpha}e^{v(r)}\geq Cr^{2+\alpha-\delta},$$
where $C$ is independent of $r$. By Jensen's inequality, we have

$$\max_{\partial B_{r}}(F^{0}(x)^{2+\alpha}e^{u(x)})=r^{2+\alpha}\max_{\partial B_{r}}e^{u(x)}\geq \frac{r^{2+\alpha}}{N\kappa_{0}r^{N-1}}\int_{\partial B_{r}}e^{u}dS\geq r^{2+\alpha}e^{v(r)}\geq Cr^{2+\alpha-\delta},$$
let $r\rightarrow\infty$, we get the contradiction.
\end{proof}


\begin{thebibliography}{CL}

\bibitem{AT}F. Almgren, J. E. Taylor, \emph{Flat flow is motion by cristalline curvature for curves with cnstalline energies,} J. Differential Geom, 42 (1995), pp. 1-22.

\bibitem{ATW}F. Almgren, J. E. Taylor, L. Wang, \emph{Curvature-driven flows: a variational approach,} SIAM J. Control Optim, 31 (1993), pp. 387-437.

\bibitem{AFTL} A. Alvino, V. Ferone, G. Trombetti, P. L. Lions, \emph{Convex symmetrization and applications}, Ann. Inst. H. Poincar\'{e} Anal. Nonlin\'{e}aire, 14 (1997), pp. 275-293.

\bibitem{AR}A. Ambrosetti, P. H. Rabinowitz, \emph{Dual variational methods in critical point theory and applications}, J. Funct. Anal. 14 (1973), pp. 349-381.

\bibitem{AY} W. W. Ao, W. Yang, \emph{On the classification of solutions of cosmic strings equation}, Ann. Mat. Pura Appl. 198 (4) (2019), pp. 2183-2193.

\bibitem{BFK}M. Belloni, V. Ferone, B. Kawohl, \emph{Isoperimetric inequalities, Wulff shape and related questions for strongly nonlinear elliptic operators}, Z. Angew. Math. Phys, 54 (5) (2003), pp. 771-783.

\bibitem{BN}H. Brezis, L. Nirenberg, \emph{Remarks on finding critical points}, Comm. Pure Appl. Math, 44 (1991), pp. 939-963.

\bibitem{BN1}H. Brezis, L. Nirenberg, \emph{Positive solutions of nonlinear elliptic equations involving critical sobolev exponents}, Comm. Pure Appl. Math, 36 (1983), pp. 437-477.

\bibitem{CS09}A. Cianchi, P. Salani, \emph{Overdetermined anisotropic elliptic problems}, Math. Ann., 345 (2009), pp. 859-881.

\bibitem{CFV}M. Cozzi, A. Farina, E. Valdinoci, \emph{Monotonicity formulae and classification results for singular, degenerate, anisotropic PDEs}, Adv. Math., 293 (2016), pp. 343-381.

\bibitem{cfv} M. Cozzi, A. Farina, E. Valdinoci, \emph{Gradient bounds and rigidity results for singular, degenerate, anisotropic partial differential equations}, Comm. Math. Phys. 331 (1) (2014), pp. 189-214.

\bibitem{CR}M. G. Crandall, P. H. Rabinowitz, \emph{Some continuation and variational methods for positive solutions of nonlinear elliptic eigenvalue problem}, Arch. Rational Mech. Anal. 58 (1975), pp. 207-218.

\bibitem{DF} E. N. Dancer, A. Farina, \emph{On the classification of solutions of $-\Delta u=e^{u}$ on $\mathbb{R}^{N}$: stability outside a compact set and applications}, Proc. Amer. Math. Soc, 137 (2009), pp. 1333-1338.

\bibitem{DG}F. Della Pietra, N. Gavitone, \emph{Sharp bounds for the first eigenvalue and the torsional rigidity related to some anisotropic operators}, Math. Nachr, 287 (2014), pp. 194-209,

\bibitem{F} A. Farina, \emph{Stable solutions of $-\Delta u=e^{u}$ on $\mathbb{R}^{N}$}, C. R. Math. Acad. Sci. Paris, 345 (2007), pp.  63-66.


\bibitem{FV14} A. Farina, E. Valdinoci, \emph{Gradient bounds for anisotropic partial differential equations}, Calc. Var. Partial Differential Equations, 49 (2014), pp.  923-936.

\bibitem{FL}M. Fazly, Y. Li, \emph{Partial regularity and Liouville theorems for stable solutions of anisotropic elliptic equations}. Preprint. Arxiv 2008.04455.

\bibitem{FK}V. Ferone, B. Kawohl, \emph{Remarks on a Finsler-Laplacian}, Proc. Amer. Math. Soc, 137 (2009), pp. 247-253.

\bibitem{FS} G. M. Figueiredo, J. R. Silva, \emph{Solutions to an anisotropic system via subsupersolution method and Mountain Pass Theorem}, Electronic Journal Quality Theory in Differential Equations, 46, (2019), pp. 1-13.

\bibitem{FM} I. Fonseca, S. M\"{u}ller, \emph{A uniqueness proof for the Wulff theorem}, Proc. Roy. Soc. Edinburgh Sect. A, 119 (1991), pp. 125-136.

\bibitem{GP}J. Garcia Azorero, I. Peral Alonso, \emph{On an Emden-Fowler type equation}, Nonlinear Analysis, 18 (11) (1992), pp. 1085-1097.

\bibitem{GPP}J. Garcia Azorero, I. Peral Alonso, J. P. Puel, \emph{Quasilinear problems with exponential growth in the reaction term}, Nonlinear Analysis, 22 (4) (1994), pp. 481-498.

\bibitem{MP}F. Mignot, J. P. Puel, \emph{Sur une class de probl\`{e}mes non lin\'{e}aires avec non lin\'{e}airit\'{e} positive, croissante, convexe}, Comm. Partial Differential Equations, 5 (1980), pp. 791-836.

\bibitem{NT}W. M. Ni, I. Takagi, \emph{On the shape of least-energy solutions to a semilinear neumann problem}, Comm. Pure Appl. Math, 44 (1991), pp. 819-851.

\bibitem{S} J. Serrin, \emph{Local behavior of solutions of quasi-linear equations}, Acta Math, 111 (1964), pp. 247-302.

\bibitem{S1}J. Serrin, \emph{On the strong maximum principle for quasilinear second order differential inequalities}, J. Funct. Anal. 5 (1970), pp. 184-193.

\bibitem{S2}G. Stampacchia, \emph{\'{E}quations elliptiques du second ordre \`{a} coefficients discontinus}, S\'{e}minaire Jean Leray, 3 (1963-1964), pp. 1-77

\bibitem{WY}C. Wang, D. Ye, \emph{Some Liouville theorems for H\'{e}non type elliptic equations}, J. Funct. Anal. 262 (4) (2012), pp. 1705-1727.

 \bibitem{WX} G. F. Wang, C. Xia,  \emph{A characterization of the Wulff shape by an overdetermined anisotropic PDE}, Arch. Rational Mech. Anal, 199 (2011), pp. 99-115.

 \bibitem{WX12} G. F. Wang, C. Xia, \emph{Blow-up analysis of a Finsler-Liouville equation in two dimensions}, J. Differ. Equations, 252 (2012), pp.  1668-1700.


\bibitem{W} G. Wulff, \emph{Zur Frage der Geschwindigkeit des Wachstums und der Auflösung der Kristallflächen}, Z. Krist, 34 (1901), pp.  449–530.


\end{thebibliography}
\end{document}